\documentclass[12pt]{article}
\usepackage[applemac]{inputenc}
\usepackage{lscape}
\usepackage[dvips]{graphicx}
\usepackage[psamsfonts]{amssymb}
\usepackage[all]{xy}
\usepackage{amsmath,amsthm}

\usepackage{color}

\newtheorem{Def}{Definition}[subsection]
\newtheorem{Defs}[Def]{Definitions}

\newtheorem{Lem}[Def]{Lemma}
\newtheorem{Pro}[Def]{Proposition}
\newtheorem{Cor}[Def]{Corollary}

\newtheorem{Rem}[Def]{Remark}

\newtheorem{Not}[Def]{Notation}

\newcommand{\lra }{ \longrightarrow}

\newcommand{\bk}{l\!k}
\newcommand{\id}{\mbox{\rm id}}

\newcommand{\bB}{\bar {\mathbb B}}
\newcommand{\hH}{H\!H}

\makeindex

\begin{document}

 \centerline{\Large On the   closed geodesics problem }
 \vspace{2mm}

\centerline{ Bitjong Ndombol \footnote{Facult\'e des Sciences, Universit\'e de Yaound\'e 1,  Cameroun; bitjongndombol@yahoo.fr} $\,$   
}

\begin{abstract}  Let $\bk $ be a field of characteristic $p\geq 0$ and $X$  a  simply connected finite  CW complex. In this text, we prove that: {\sl if  the cohomology algebra $H^*(X;\bk)$ is generated, as an algebra, by at least two linearly independent  elements, then the  sequence  of Betti numbers $ \left( \dim H^n(LX;\bk)\right)_{n\geq 1 }$
grows unbounded.} This provides a complete solution  of the closed geodesics problem.
\end{abstract}
\vspace{5mm}

  \noindent {\bf AMS Classification} : 
16E40 (primary)  16E45,  55P35 , 55Pxx,  57T30,  57T05 (secondary).

    \vspace{2mm}
    \noindent {\bf Key words} : (co)bar construction, closed geodesics, Hochschild homology,  free loop space, loop suspension, spectral sequence.

\vspace{5mm}
\section{Introduction.} 
 \emph{ When does a simply connected closed 
Riemannian manifold admit infinitely many geometrically distinct  closed geodesics? } This is the closed geodesics problem in Riemannian Geometry.

 An  answer is  given by the following  theorem of  Gromoll and Meyer. 
 \vspace{3mm}
 
 For any space $X$, let $LX :=\mbox{Map}(S^1,X)$ be  the free loop space  equipped  with the compact-open topology.

 \paragraph{Theorem     \cite{GM}}{\sl Let $M$ be a simply connected closed Riemannian manifold. If there exists a field $\bk$  such that the sequence of Betti numbers $(\mbox{dim}H_n(LM; \bk))_{n\geq1}$ is unbounded, then, for each Riemannian metric, there exist infinitely many geometrically distinct closed geodesics on $M$.}
 
  In the work presented here, we are mainly concerned  with  the 

{\bf  Topological conjecture:} \emph{  Let $\bk $ be a field and $X$  a  simply connected  finite CW complex. The sequence  of Betti numbers $ \left( \dim H^n(LX;\bk)\right)_{n\geq 0 }$
is unbounded if and only if  the cohomology algebra $H^*(X;\bk)$ is generated, as an algebra, by at least two linearly independent  elements.}
 
 The proof that, {\sl If the sequence  of Betti numbers $ \left( \dim H^n(LX;\bk)\right)_{n\geq 0 }$
is unbounded, then the cohomology algebra $H^*(X;\bk)$ is generated, as an algebra, by at least two linearly independent  elements} is obvious (See \cite{HV}, Introduction). The converse    has been proved for the following particular  cases:
\begin{enumerate}
\item $X$ is a simply connected CW complex of finite type and   $\bk = \mathbb Q$ (\cite{VS}),
\item $\bk=\mathbb F_p$ and $X$ is  $\mathbb F_p$-elliptic  (\cite{JM}), 
\item $X$ is $\bk$-formal  where $\bk$ is any field  (\cite{HV}),
\item $X$ is $n$-connected and $\mbox{char}(\bk)\geq \mbox{dim}(X)/n$   (\cite{HV}),
\item $X$ is a homogeneous space and $\bk$ is any field (\cite{MZ}),
\item  $X$ is   a connected sum $M_1\#M_2$  of two manifolds $M_1$ and $M_2$   when  $\bk$  is a  field such that neither  $M_1$ nor $M_2$ has the same cohomology with coefficients in $\bk$ as a sphere \cite{Lam}.
\end{enumerate} 
\vspace{3mm}
 In this text we prove this topological conjecture, that is:
  \paragraph{Main  Theorem } {\sl  If $X$ is a simply connected finite CW complex such that the cohomology algebra   $H^\ast(X;  \bk)$  of $X$ has at least two linearly independent generators as an algebra,
then the sequence $(dim H^n(\Omega X;\bk))_{n\geq 1}$ of the Betti numbers of $\Omega X$, the  loop space of $X$ and the sequence  $(dim H^n(LX;\bk))_{n\geq 1}$ of the Betti numbers of $LX$, the free loop space of $X$,  grow unbounded. }

\vspace{3mm}

The unbounded  growth of the Betti numbers of $\Omega X$ is proved in \cite{Mc} with the use of Bockstein spectral sequences and $\infty$-implications.
\vspace{3mm}
\paragraph{ Corollary} {\sl  Let $X$ be a simply connected closed Riemannian manifold. If $H^*(X; \bk)$ is generated, as an algebra, by at least two linearly independent elements, then for each Riemannian metric there exist infinitely many geometrically distinct closed geodesics on $X$.}
\vspace{3mm}

To solve the {\bf Topological conjecture}, we need auxiliary results. For this purpose,  we recall some definitions.

Suppose that  the cochain complex  $ ( C, d_C)$    be  the mod$_p$ reduction of the cochain complex $( \hat C, d_{\hat C})$ with coefficients in $\mathbb Z$.

\vspace{5mm}

We recall  the universal coefficients theorem 
 $$H^n( C, d_C) \cong H^n(\hat C, d_{\hat C})\otimes \mathbb F_p \oplus s\mbox{Tor}(H^{n+1}(\hat C, d_{\hat C}), \mathbb F_p) = H_0^n( C, d_C) \oplus H_1^n( C, d_C)$$ where $s$ is the suspension and $\mbox{Tor}(H^{n+1}(\hat C, d_{\hat C}), \mathbb F_p)$ is the mod$_p$ reduction of the submodule of  $p$-torsion  of  $H^{n+1}(\hat C, d_{\hat C})$.
  
  \vspace{3mm}
  
 Throughout  this text, $(A, d_A)$ is   the DG algebra of normalized singular cochains on  a finite CW complex  $X$  with coefficients in the prime field $\mathbb F_p$.

 Thus $(A, d_A)$ is such that $A:=\{A^k\}_{k\geq 0}$ and  $\deg d = + 1$  satisfying the  hypothesis 
 \begin{eqnarray}
 H^0(A, d_A)= \mathbb F_p\,, \quad H^1(A, d_A)=0 \,, \quad \, \dim H^i(A, d_A)<\infty\,, H^i(A, d_A)=0, i\geq n.
\end{eqnarray}

 \vspace{2mm}
 \paragraph{ Theorem 1} {\sl Let  $(A, d_A)$ be   the DG algebra of normalized singular cochains on  a finite CW complex   $X$ with coefficients in the prime field $\mathbb F_p$.
 
 If $\beta\in H_1^*(A, d_A)\smallsetminus\{0\}$, $\beta=s\alpha$, such that  
 
  $\sigma(\alpha)\neq 0$,
  then
 there exist two infinite sequences  $(\chi_n)_{n\geq 1}$ and $(\psi_n)_{n\geq 1}$  in $H^*\bB (A, d_A)\smallsetminus\{0\}$ such that for every $n\geq 1$, $\psi_n\in H^{even}\bB (A, d_A)$ and  $\chi_n\in H^{odd}\bB (A, d_A)$.}
 
 \vspace{2mm}

 \vspace{2mm}
 \paragraph{ Theorem 2} {\sl Let  $X$ be   a finite  CW complex  $X$  with coefficients in the prime field $\mathbb F_p$.
 If $H^*_1(X; \mathbb F_p)\neq \{0\}$, then:
\begin{enumerate}
\item the sequence  $(dim H^n(\Omega X, \mathbb F_p))_{n\geq 1}$ grows unbounded where $\Omega X$ is the loop space of $X$,  
 \item the sequence $(H^n(LX; \mathbb F_p))_{n\geq 1}$ grows unbounded where  $LX$ is the free loop space of $X$. 
\end{enumerate}}

\vspace{3mm}
Section 2 is a recollection  of algebraic settings. 
  Theorem 1 is proved in section 3 while the proof of Theorem 2 is provided in Section 4.

 In  section  5, we apply  these algebraic  results  to   prove the main  Theorem.

\vspace{3mm}

In this paper,  results are over the field  $\bk$ of characteristic $p\geq 0$, even if the proofs are done    with  the prime field  $\mathbb F_p$, $p\geq 2$.

 \vspace{3mm}

   \tableofcontents 


  \section{ Algebraic preliminaries} 
In this section, $\bk$ is a commutative ring with unit.

 \subsection{The Adams reduced bar construction}

   Let  $(A, d)$ be a  supplemented DG algebra, that is $A=\bk \oplus \bar A$.
    
    The \emph{Adams reduced  bar construction}, $\bB (A, d_A):=\left(\left\{ \bB_k(A)\right\}_{k\geq 0}, \delta\right)$ is defined as follows:
    \begin{enumerate}
 \item  $\bB _k (A) =   
    (s\overline A)^{\otimes k} $ where $(sA)^{i+1} = A^i$. 
        \item A generic element of  $\bB _k   A$   is written as a sum of elements of the form $[a_1|a_2\cdots|a_k]$ with
     $$ \begin{array}{ll}
     \deg [a_1|a_2\cdots|a_k] &=  \sum_{i=
    1}^k \deg s a_i \\
    &=  \sum_{i=
    1}^k \deg a_i   - k \,.
    \end{array} $$ 
    \item  The differential $\delta = \delta'+\delta''$ is defined by
    the two  differentials,
\begin{eqnarray}
 \left\{ 
 \begin{array}{llll}
\delta':  \bB_k A\rightarrow   \bB_k A\,, \quad & \delta'[a_1|\cdots |a_k] &= -\sum_{i=1}^k(-1)^{\epsilon_i}[a_1|..|d_Aa_i|\cdots|a_k]\\
\delta'':  \bB_k A\rightarrow   \bB_{k-1} A\,, \quad& \delta''[a_1|\cdots|a_k] &= \sum_{i=1}^k(-1)^{\epsilon_i}[a_1|\cdots|a_{i-1}a_i|\cdots|a_k]
\end{array}
\right.
\end{eqnarray}
\end{enumerate}
where $\epsilon_1=0$ and $\epsilon _i = \deg(sa_1)+\deg(sa_2)+\cdots +\deg(sa_{i-1})$,  $i\geq 2$.

 \vspace{3mm}

 \subsection{Adams reduced cobar construction}

   Let  $\mathcal C:=(C, d_C)$ be a differential graded coalgebra   (for a definition of a DG coalgebra see \cite{M}).
   
  Assume that $\mathcal C$ is such that, 
 $$C=\{C^i \}_{i\in \mathbb N} \,, \qquad d_C: C^i \to C^{i+1}\,.$$
 Let  $\varepsilon : C \to \bk$ be the co-unit of $   C $ and $\Delta_C : C \to C\otimes C$ the coproduct of $C$.
 We also assume that the DG coalgebra $C$ is coaugmented,  that is, there exists a homomorphism of DG coalgebras 
 $$\eta : \bk \to C$$
 called  a coaugmentation. We set 
 $
 \bar C = \mbox{coker}\, \eta\,
 $ and write, by abuse,  $ C = \bk  \oplus \overline C$ identifying $\eta(1)$ with 1. 
 
 \vspace{2mm} Observe that:
$$
\begin{array}{l}
 \renewcommand{\arraystretch}{1.6}
\left\{
\begin{array}{ll}
\Delta_C(1)&=1\otimes 1\\
\Delta_C (\bar C) &\subset \bar C \otimes C + C \otimes \bar C= (\bar C \otimes \bk )\oplus (\bk \otimes \bar C) \oplus (\bar C \otimes \bar C)\\
\end{array} 
 \right.
\\
d_C (\bar C) \subset \bar C
 \,.
 \renewcommand{\arraystretch}{1}
  \end{array}
  $$
 We  shall write, for $c \in \bar C$:
$$
 \Delta_C (c) = c\otimes 1 + 1 \otimes c + \bar \Delta_C c\,, \qquad \bar \Delta _C = \displaystyle \sum _{i} c'_i\otimes c''_i  \quad \mbox{ with }\quad  c'_i , c''_i \in \bar C\,,
 $$
 where $\bar \Delta $ is called the \emph{reduced coproduct}.
 
  \vspace{2mm} 
   Before entering upon  the definition of the  \emph{reduced  Adam's cobar construction}, let us introduce suitable  notations to  describe elements in the  tensor algebra $T^a(s^{-1} V)$ freely generated by the \emph{desuspension} of a graded vector space $V$.
   
  A generic element of  $T^a(s^{-1} V)$ is a sum of monomial elements,
    $$
    \langle v_1|v_2|\dots |v_k\rangle   
    := \left\{ 
    \begin{array}{ll}
    1\in \bk \cong T^0(s^{-1} V) &\mbox{ if } k=0\\
    s^{-1} v_1\otimes s^{-1}v_2\otimes \dots s^{-1}v_k  \in T^k(s^{-1}\bar V)  &\mbox{ if } k\geq 1
    \end{array}
    \right.\,,
    $$
  and 
  $$
  \deg \langle v_1|v_2|\dots |v_k\rangle =  \sum_{i=1}^k \deg s^{-1}v_i= 
  \sum_{i=1}^k \deg v_i + k\,.
  $$

        The  {\it normalized} cobar construction  of  the DG coalgebra $\mathcal C=(C,d_C)$   is the  DG  algebra $\bar \Omega(\mathcal C)= \left(T^a(s^{-1}\overline C),  d:=d'+ d'' \right)$ 
   with      
     $$
      \renewcommand{\arraystretch}{1.6}
    \begin{array}{ll}
   d' \langle c_1c_2\vert\cdots \vert c_{k}\rangle&=
 \displaystyle\sum_{j=1}^{k}
   (-1)^{\varepsilon_i } \langle
   c_1|c_2\vert\cdots\vert d_Cc_i\vert. \cdots\vert c_{k}\rangle \\
    d''\langle c_1c_2\vert\cdots \vert c_{k}\rangle&=
 \displaystyle\sum_{i=1}^{k}\sum_{j\in J}
   (-1)^{\varepsilon_i+\vert c'_{i_j}\vert}\, \langle
   c_1|c_2\vert\cdots\vert c'_{i_j}\vert c''_{i_j}\vert\cdots\vert
   c_{k}\rangle 
   \end{array}
 \renewcommand{\arraystretch}{1}
   $$
  with $\varepsilon_1=0$,  $\varepsilon_i = \deg(   c_1)+\deg( c_2)+\cdots+
  \deg( c_{i-1})+i-1 $ for $i\geq 2$ and  $\bar \Delta c_i=\sum_{j\in J}  
c'_{i_j}\otimes c''_{i_j}$ as in formula (3).
  
\vspace{2mm}
 Recall that  a coaugmented   graded  coalgebra $C$  
is  \emph{locally  conilpotent} if for each   $x
 \in \overline{C}$, there exists a positive integer  $k$ such that   the reduced coproduct $\overline{\Delta}^{(k)}
 x = 0$.   For instance,  the \emph{free coalgebra}  $T(sV)$, such that 
$V^{< 0} = 0$ is   a locally  conilpotent  DG coalgebra for the coproduct  
$$
 \Delta \langle v_1| v_2| \cdots | v_k\rangle=
 \displaystyle \sum_{i= 1}^{k-1}  \langle v_1|  \cdots | v_i\rangle
 \otimes \langle v_{i+1} | \cdots | v_k\rangle \,.
$$
 Therefore, if $(A, d_A)$ is a cochain algebra, then  $\bB(A, d_A)$ is a 1-reduced  conilpotent  DG coalgebra.
 
The reduced coproduct on the DG coalgebra $\bB (A, d_A)$ is given by 
\begin{eqnarray}
 \Delta [ a_1| a_2| \cdots | a_k]=
 \displaystyle \sum_{i= 1}^{k-1}  [ a_1|  \cdots | a_i]
 \otimes [ a_{i+1} | \cdots | a_k] \,.
\end{eqnarray}

 When  ${ \sc DGA}$ (resp. ${\sc NDGC}$)  denotes the category of  DG algebras (respectively locally conilpotent  DG-coalgebras), 
 the functors 
$$
\bB : {\sc DGA}  \lra  {\sc NDGC}  \quad \mbox{and} \quad \bar \Omega :   {\sc NDGC} \lra  {\sc DGA}$$ are \emph{adjoint functors} (see \cite[Proposition 2.11]{ACC}).

We  denote by $\alpha_ A$ the unit of this adjunction, that is,  we have  the surjective homomorphism of DG algebras 
\begin{eqnarray}
\xymatrix{
\overline{ \Omega} \bB  (A, d_A) \ar@{->>}[rr]^{\alpha_{A}}&& (A, d_A)}\,, \quad  <[a_1|a_2|\dots|a_n]> \mapsto \left\{ 
\begin{array}{ll} 
1 &\mbox{ if } n= 0\\
 a_1 &\mbox{ if  } n=1 \\
  0 &\mbox{ otherwise} \\
\end{array}
\right.\,. 
\end{eqnarray}
      It is well known that $\alpha_A$ is a quasi-isomorphism (see  \cite[Proposition 2.14]{ACC} or \cite[Corollary 2.15]{Mu}).

 \vspace{3mm}

\subsection{Further algebraic tools} 
  \begin{Def}  Let $(A, d_A)$ be a DG algebra. A \emph{Kraines   sequence} of length $n$  starting at $a_1\in A^{odd} $ is a  
  sequence  $a_1, a_2, \cdots , a_{k}, \cdots a_n$  which satisfies, for all $1\leq k\leq n$,
\begin{eqnarray}
d_A a_k = 
\left\{ 
\begin{array}{cl}
0 &\mbox{ if } k=1\\
\sum _{j=1} ^{k-1}  a_j a_{k-j}  &\mbox{ if } k\geq 2\,.
\end{array} \right.
\end{eqnarray}
If a sequence $(a_k)_{k\geq 1}$ satisfies (5) for every $k\geq 1$, it is an infinite \emph{Kraines sequence}.
 \end{Def}

\vspace{3mm}
 The term \emph{Kraines sequence} refers to the sequence defined by  D. Kraines  (\cite{Kr1}, Definition 11) with a different  sign.

\begin{Def}(\cite{Mc2} Definition 8.14)  A  DG algebra with \emph{cup one product}  consists of an $\bk$-dga $A=\{A^k\}_{k\geq 0}$  endowed with  a degree -1  $\bk$-linear map, 
$$
\smile_1: A  \otimes  A\lra  A\,, 
 $$
 which vanishes  in degree 0 and satisfies the identities
 \begin{eqnarray}
d_A (a\  \smile_1b )= ab  - (-1)^{\mbox{deg}(a)\mbox{deg}(b)}  ba - (d_A a) \smile_1 b
 -(-1^{\mbox{deg} (a)}  a  \smile_1 d_Ab\,,
 \end{eqnarray} 
 \begin{eqnarray}
(ab)\smile_1c = (-1)^{\mbox{deg}(a)}a(b\smile_1c) + (-1)^{\mbox{deg}(b)\mbox{deg}(c)} (a\  \smile_1 c)b\,.
\end{eqnarray}  
\end{Def}

\vspace{2mm}
\begin{Lem}(\cite{S-E}){\sl Let $X$ be a topological space and $\bk$ a commutative ring with unit. The  DG algebra $(A, d_A) =C^*(X; \bk)$ of normalized singular cochains on $X$  with coefficients in $\bk$ has a \emph{cup-one product} .}\end{Lem}

\vspace{3mm} 

   \begin{Not}\begin{enumerate}
   
   \item  In the following,  $(\hat A, d_{\hat A})= C^*(X; \mathbb Z)$ is the DG algebra of normalized singular  cochains on the finite CW complex $X$ with coefficients in $\mathbb Z$ and $ (A, d_A)=(\hat A, d_{\hat A})\otimes_{\mathbb Z}\mathbb F_p= C^*(X; \mathbb F_p)$, the DG algebra of normalized singular cochains on $X$ with coefficients in $\mathbb F_p$. We note  $\mbox{red}_p: \hat A\to A$  the reduction mod$_p$.
   \item We denote by  $\mbox{tor}H^*(\hat A, d_{\hat A})$  the $p$-torsion part of $H^*(\hat A, d_{\hat A})$
   \item Let $\epsilon\geq 1$ a positive integer.
   We note  $(A_{\epsilon}, d_{A_{\epsilon}})=(\hat A, d_{\hat A})\otimes_{\mathbb Z}\mathbb Z/p^{\epsilon}\mathbb Z$ and $\mbox{red}_{p^{\epsilon}}:(\hat A, d_{\hat A})\to (A_{\epsilon}, d_{A_{\epsilon}})$ the projection.
    \item  The homomorphism  $\beta_{\epsilon}: H^q(A_{\epsilon}, d_{A_{\epsilon}})\to H^{q+1}(A_{\epsilon}, d_{A_{\epsilon}})$ is the Bockstein homorphism associated to the short exact sequence $0\to \mathbb Z/p^{\epsilon}\mathbb Z\to \mathbb Z/p^{2\epsilon}\mathbb Z\to \mathbb Z/p^{\epsilon}\mathbb Z\to 0$.

   \item The  differential on $\bB(\hat A, d_{\hat A})$ is denoted $\hat \delta$ while the differential  on $\bB\mathcal A$ is denoted $ \delta$. 
  \item The coproduct is denoted $\Delta$ on $\bB A$, $\bB\hat A$ , $H^*\bB A$ and on $H^*\bB(\hat A, d_{\hat A})$.
  \item If $x\in  A \hspace{2mm}  (resp. \hat A, \bB A$, $\bB\hat A$) is a cocycle, the cohomology class of $x$ is denoted $\mbox{cls}(x)$. \end{enumerate}
   \end{Not} 
   \vspace{5mm}
We recall  the universal coefficients theorem 
  \begin{eqnarray}H^n(A, d_A) \cong H^n(\hat A, d_{\hat A})\otimes \mathbb F_p \oplus s\mbox{Tor}(H^{n+1}(\hat A, d_{\hat A}), \mathbb F_p) = H_0^n(A, d_A) \oplus H_1^n(A, d_A)\end{eqnarray} where $s$ is the suspension and $\mbox{Tor}(H^{n+1}(\hat A, d_{\hat A}), \mathbb F_p)$ is the mod$_p$ reduction of the submodule of  $p$-torsion  of  $H^{n+1}(\hat A, d_{\hat A})$.
  
  \vspace{3mm}
  
  \begin{Rem}  Let  $\beta\in H^*_1(A, d_A)$. The following assertions are equivalent:\begin{enumerate}
  \item  $\beta=s\alpha$, $\alpha\in H^*_0(A, d_A)\smallsetminus\{0\}$,
  \item  there exist $\hat b\in \hat A$ ,  an integer $\epsilon\geq 1$ such that $d_{\hat A}\hat b= p^{\epsilon}\hat a$ with $\mbox{cls}(\mbox{red}_p(\hat a))=\alpha$ and $\mbox{cls}(\mbox{red}_p(\hat b))=\beta$,
  \item $\beta_{\epsilon}(\mbox{cls}(\mbox{red}_{p^{\epsilon}}(\hat b)))= \mbox{cls}(\mbox{red}_{p^{\epsilon}}(\hat a)).$ 
  \end{enumerate} 
\end{Rem}
 \vspace{3mm}  
  \begin{Def} The cohomology loop suspension 
 is the graded linear map   
$$ \sigma =\{\sigma^n\}_{n\geq 1}, 
\sigma^{n+1} : QH^{n+1}(A, d_A)\lra PH^n(\bB (A, d_A))\,, \qquad \alpha \mapsto \mbox{cls} ( [a])
$$
where  $a \in A^{n+1}$ is a  cocycle representing $\alpha$, $QH^{n+1}(A, d_A)$ the indecomposable elements of $H^{n+1}(A, d_A)$ and $ PH^n(\bB (A, d_A))$ the primitive elements of  $H^n(\bB (A, d_A))$ .\end{Def}
   
\begin{Lem}{\sl Suppose that $(A, d_A)$ is the mod$_p$ reduction of the DG algebra $(\hat A, d_{\hat A})$ with coefficients in $\mathbb Z$ and let  $\beta \in H_1^*(A, d_A)$ such that $\beta =s\alpha$. 

If
$\sigma(\alpha)\neq 0$   then  $\sigma(\beta)\neq 0$.
}\end{Lem}
\begin{proof}Since $\beta =s\alpha$, there exist $\hat b, \hat a\in \hat A$ and an integer $\epsilon\geq 1$ such that $d_{\hat A}\hat b= p^{\epsilon}\hat a$ with  $\beta= \mbox{cls}(\mbox{red}_p(\hat b))$ and $\alpha= \mbox{cls}(\mbox{red}_p(\hat a))$.

 Set $a=\mbox{red}_p(\hat a)$, $b=\mbox{red}_p(\hat b)$ and remark that $$\hat \delta [\hat b] = p^{\epsilon}[\hat a ]\quad \mbox{with}\quad \bB\mbox{red}_p([\hat a ])=[a].\quad (\star)$$

 Suppose that $\sigma(\alpha)=\mbox{cls}([a])\neq 0$. The equality $(\star )$ means that  $\sigma(\beta)= \mbox{cls}([b])=s\mbox{cls}([a])=s\sigma(\alpha)\neq 0$.

\end{proof}

\begin{Lem}{\sl  Let $\hat x\in \hat A$ such that  $d_{\hat A}\hat x=p^{\epsilon }\hat z$ with  $\mbox{red}_p(\hat z)\neq 0$, where $\epsilon\geq 1$ is an  integer. 

If  $\mbox{cls}(\mbox{red}_p(\hat x))\neq 0$ and  $\beta_{\epsilon}(\mbox{cls}(\mbox{red}_{p^{\epsilon}}(\hat x)))=\mbox{cls}(\mbox{red}_{p^{\epsilon}}(\hat z))=0$,  then there exists $\hat \zeta\in \mbox{ker}\hspace{1mm}\mbox{red}_p$ such that $d_{\hat A}(\hat x+\hat \zeta)=0$.

}\end{Lem}
\begin{proof}It is enough to prove that $\mbox{cls}(\mbox{red}_p(\hat x))\in H^*_0(A, d_A)$.

 Suppose that $\mbox{cls}(\mbox{red}_p(\hat x))\in H_1^*(A, d_A)$. Then $\mbox{cls}(\mbox{red}_p(\hat x))=s\mbox{cls}(\mbox{red}_p(\hat z))= sH^*(\mbox{red}_p)(\mbox{cls}(\mbox{red}_{p^{\epsilon}}(\hat z))=0$. This is impossible, since $\mbox{cls}(\mbox{red}_p(\hat x))\neq 0$. 
 Thus  $\mbox{cls}(\mbox{red}_p(\hat x))\in H^*_0(A, d_A)$ and there exists $\hat \zeta\in \mbox{ker}\hspace{1mm}\mbox{red}_p$ such that $d_{\hat A}(\hat x+\hat \zeta)=0$.

\end{proof}
\begin{Rem} Let $\epsilon\geq 1$ be an integer and $\rho_A:\hat A\to \hat A/p^{\epsilon}\hat A$ the projection. \begin{enumerate}

\item The product  on $\hat A$ yields a product  on $\hat A/p^{\epsilon}\hat A$ given by  $\rho_A(\hat  x)\rho_A(\hat  y)= \rho_A(\hat x \hat y)$, for $\hat x, \hat y\in \hat A$, such that $\bar d_{\hat A}$ is a derivation where $\bar d_{\hat A}$ is the differential induced on $\hat A/p^{\epsilon}\hat A $ by $d_{\hat A}$. Hence $(\hat A/p^{\epsilon}\hat A, \bar d_{\hat A})$ is a DG algebra.
 \item Observe that the cup one product $\smile_1$ on $(\hat A, d_{\hat A})$ is $\mathbb Z$- linear and hence induces a
  cup one product  also noted $\smile_1$ on $(\hat A/p^{\epsilon}\hat A, \bar d_{\hat A})$  given by   $\rho_A(\hat  x)\smile_1\rho_A(\hat  y)= \rho_A(\hat x \smile_1\hat y)$ for $\hat x, \hat y\in \hat A$.
  
   The latter induces a cup one product $$\smile_1:((\hat A/p^{\epsilon}\hat A\otimes _{\mathbb Z}\mathbb Q)\otimes_{\mathbb Q} ((\hat A/p^{\epsilon}\hat A)\otimes _{\mathbb Z}\mathbb Q)\to (\hat A/p^{\epsilon}\hat A)\otimes _{\mathbb Z}\mathbb Q.$$\end{enumerate}\end{Rem}
\begin{Lem}{\sl Let  $\epsilon\geq 1$ be an integer.\begin{enumerate}
\item If $\hat x$ is a cocycle in $(\hat A, d_{\hat A})$ such that $\mbox{cls}(\rho_A(\hat x))\in \mbox{tor}H^*(\hat A/p^{\epsilon}\hat A, \bar d_{\hat A})\smallsetminus\{0\}$, then $\mbox{cls}(\hat x)\in \mbox{tor}H^*(\hat A,d_{\hat A})$ where $\rho_A:\hat A\to \hat A/p^{\epsilon}\hat A$ is the projection.
\item If $(\bar a_n)_{1\leq n\leq N}$ is a \emph{Kraines sequence} in $\hat A/p^{\epsilon}\hat A$, then $\mbox{cls}(\sum_{i=1}^N\bar a_i\bar a_{N+1-i})\in \mbox{tor}H^*(\hat A/p^{\epsilon}\hat A, \bar d_{\hat A})$.
\end{enumerate}

}\end{Lem}
\begin{proof}\begin{enumerate}
\item Suppose $\mbox{cls}(\rho_A(\hat x))\in \mbox{tor}H^*(\hat A/p^{\epsilon}\hat A, \bar d_{\hat A})\smallsetminus\{0\}$, then there exists an integer $\eta\geq 1$  such that $p^{\eta}\mbox{cls}(\rho_A( \hat x))=0$. 
\begin{enumerate}
\item If $1\leq\eta<\epsilon$, there exists $\bar y\in \hat A/p^{\epsilon}\hat A$ such that $\bar d_{\hat A}\bar y=p^{\eta}\rho_A(\hat  x)$. Let $\hat y\in \hat A$ such that $\rho_A(\hat y)=\bar y$, there exists $\hat \zeta\in p^{\epsilon}\hat A$ such that $d_{\hat A}\hat y=\hat x+\hat \zeta$.

Set $\hat \zeta=p^{\epsilon }\hat z$, $\hat z\in \hat A$, then  $d_{\hat A}\hat y=p^{\eta}(\hat x+p^{\epsilon-\eta}\hat z)$ and hence $\mbox{cls}(\mbox{red}_p(\hat y))=s\mbox{cls}(\mbox{red}_p(\hat x))$. Thus $\mbox{cls}(\hat x)\in \mbox{tor}H^*(\hat A,d_{\hat A})$.
\item Suppose that   $\eta\geq \epsilon$. The equality $p^{\eta}\mbox{cls}(\rho_A(\hat y))=0$ means that there exists a cocycle $\bar y\in \hat A/p^{\epsilon}\hat A$ such that $$\beta(\mbox{cls}(\bar y))=p^{\eta}\mbox{cls}(\hat x)\in H^*(p^{\epsilon}\hat A, d_{\hat A})\qquad (\star)$$ where $\beta:H^*(\hat A/p^{\epsilon}\hat A, \bar d_{\hat A})\to H^{*+1}(p^{\epsilon}\hat A, d_{\hat A})$ is the Bockstein homomorphism associated to the short exact sequence 

$$0\to p^{\epsilon}\hat A\to \hat A\stackrel{\rho_A}\rightarrow \hat A/p^{\epsilon}\hat A\to 0.$$ 
 Consider
$\hat y'\in \hat A$ such that $\rho_A(\hat  y')=\bar y$. Then $d_{\hat A}\hat y'= p^{\eta'}\hat  z$, $\eta'\geq \epsilon$  and $\beta(\mbox{cls}(\bar y))=p^{\eta'}\mbox{cls}(\hat z).$

From ($\star$), we deduce that $p^{\eta'}\hat z=p^{\eta}\hat x+d_{\hat A}\hat t$ where $\hat t\in p^{\epsilon}\hat A$.  If we set $\hat y=\hat y'-\hat t$, then  $d_{\hat A}\hat y= p^{\eta}\hat x$ and  $\mbox{cls}(\hat x)\in \mbox{tor}H^*(\hat A, d_{\hat A})$.
\end{enumerate}

\item 
    Since the DG algebra $ (\hat A/p^{\epsilon}\hat A)\otimes _{\mathbb Z}\mathbb Q$ with coefficients in $\mathbb Q$ is equipped with a cup one product, we apply Theorem 15 of \cite{Kr2}. Thus $\mbox{cls}(1\otimes_{\mathbb Z}\mathbb Q(\sum_{i=1}^Na_ia_{N+1-i}))=0$ and $\mbox{cls}(\sum_{i=1}^Na_ia_{N+1-i})\in \mbox{tor}H^*(\hat A/p^{\epsilon}\hat A, \bar d_{\hat A}).$\end{enumerate}
\end{proof}

\subsection{On minimal models}
Let $ (A, d_A)$ be the $\mbox{mod}_p$ reduction of the DG algebra $(\hat A, d_{\hat A})$ with coefficients in $\mathbb Z$  such that $$H^0(\hat A, d_{\hat A})= \mathbb Z\,, \quad H^1(\hat A, d_{\hat A})=0 \,, \quad \mbox{rank} H^i(\hat A, d_{\hat A})<\infty\,\quad \mbox{for all} \quad i.$$

By a result of \cite{HK}, there exists a  quasi-isomorphism of DG-algebras 

$$\varphi_{\hat A } :(T\hat V, d_{\hat V})\lra (\hat A, d_{\hat A})
$$
in which $T\hat V$ denotes the tensor algebra on the graded $\mathbb Z$-module $\hat V$. Moreover 
$$
\left\{
\begin{array}{ll}
\hat V= \{\hat V^i\}_{i\geq 2} \\
\mbox{rank} \hat V^i <\infty \mbox{ for } i\geq 2\\
d_{\hat V} (\hat V)\subset T^{\geq 2}\hat V\oplus p\hat V 
 \end{array}
  \right.
$$
where $T^k(\hat V)$ denote the $k^{\scriptsize \rm th}$ tensor power $\hat V^{\otimes k}$. Such a quasi-isomorphism  is unique up to an isomorphism of DG-algebras. It is called  the minimal model of the DG-algebra $(\hat A, d_{\hat A})$.

According to \cite{BL},\cite{HL}, \cite{FHT2}(page 279),   there exists a  quasi-isomorphism of DG-algebras 

$$\varphi_A :(TV, d_V)\lra ( A, d_ A)
$$
in which $TV$ denotes the tensor algebra on the graded $\mathbb F_p$-vector space $V$. Moreover we may suppose that
$$
\left\{
\begin{array}{ll}
V= \{V^i\}_{i\geq 2} \\
\dim V^i <\infty \mbox{ each } i\geq 2\\
d_V(V)\subset T^{\geq 2} V
 \end{array}
  \right.
$$
where $T^k(V)$ denote the $k^{\scriptsize \rm th}$ tensor power $V^{\otimes k}$. Such a quasi-isomorphism  is unique up to an isomorphism of DG algebras and  is called  the minimal model of the DG algebra $(A,d_A)$. 

It is well-known by \cite{ HK}  that,  if 
$$\varphi_{\hat  A} :(T\hat V, d_{\hat V})\lra (\hat A, d_{\hat A}) \quad \mbox{and}\quad \varphi_A :(TV, d_V)\lra \mathcal A$$ are minimal models, then 
 \begin{eqnarray} TV\cong T\hat V\otimes_{\mathbb Z}\mathbb F_p\quad \mbox{and} \quad  V\cong \hat V\otimes_{\mathbb Z}\mathbb F_p.\end{eqnarray}

\vspace{3mm}
Suppose that $(A, d_A)$ is the mod$_p$ reduction of the DG algebra $ (\hat A, d_{\hat A})$ with coefficients in $\mathbb Z$.

Let $\varphi_{\hat A} :(T\hat V, d_{\hat V})\lra (\hat A, d_{\hat A})$ be a minimal model. Set $V= \hat V\otimes_{\mathbb Z}\mathbb F_p$, $d_V=d_{\hat V}\otimes_{\mathbb Z}\mathbb F_p$, then  \begin{eqnarray}\varphi_A= \varphi_{\hat A}\otimes_{\mathbb Z}\mathbb F_p: (T\hat V, d_{\hat V})\otimes_{\mathbb Z}\mathbb F_p=(TV, d_V)\lra (\hat A, d_{\hat A})\otimes_{\mathbb Z}\mathbb F_p=( A, d_A) \end{eqnarray} is a quasi-isomorphism
 (\cite{ HK} (see also \cite{Lam}, proposition 15), that is $\varphi_A$ is a minimal model.  In particular, the diagram
\begin{eqnarray}\begin{array}{lll}
 (T\hat V, d_{\hat V})&\stackrel{\varphi_{\hat A}}\rightarrow&(\hat A, d_{\hat A})\\
 \mbox{red}_p\downarrow&&\downarrow \mbox{red}_p\\
 (TV, d_V)&\stackrel{\varphi_A}\longrightarrow&(A, d_A)
\end{array} \end{eqnarray}is commutative. 
\vspace{5mm}

The next result provides a  theoretical  construction of the minimal model of a DG algebra $(A, d_A)$   from  the DG algebra $\Omega \bB \mathcal A$ considered in section 2.

\begin{Pro}
Let $(A, d_A)$ be a supplemented DG algebra  as in (2), $A=\mathbb F_p\oplus \bar A$,
 and  let  $\alpha_A: \Omega \bB (A, d_A) := (T U, D) \lra (A, d_A)$ be  the  quasi-isomorphism  of DG algebras defined in paragraph 2.1.  Here    $U = s^{-1} T^{+}s\bar  A $ and   $D=\delta_0+\delta_1$ with $\delta_0 U\subset U$ and   $\delta_1 U\subset T^{\geq 2}U$. 

There exists a commutative diagram in the category of DG algebras,

\begin{eqnarray}
 \xymatrix{
(TU, D)=\Omega\bB (A, d_A)  \ar[d]_\simeq ^{\alpha_A} \ar@{->>}[rr]_\simeq ^{p_V} && (TV, d_V)\ar@{=}[d]\\
(A, d_A)&& \ar[ll]^{\varphi_A  } (TV, d_V)\ar[llu]^{\varphi_V}
 }  \,,
       \end{eqnarray}
 where $p_V$ is a  canonical surjective quasi-isomorphism and $\varphi_V$  a homomorphism of DG  algebras which  is a section of $p_V$ and hence a quasi-isomorphism. 
 
 Furthermore,   $V= \{V^i\}_{i \geq 2}$ and it is isomorphic to $H(U, \delta_0) \cong
s^{-1} H^+\bB (A, d_A)$,
\item $\varphi_A:(TV, d_V)\to (A, d_A) $ is a minimal model.
  \end{Pro}
 \begin{proof} Set $U= \mbox{ker}   \delta_0 \oplus S$ and $\ker \delta_0= V\oplus \delta_0 S$, then $U= V\oplus S\oplus \delta_0 S$.

 Let $I$ denote the ideal generated by $S$ and  $DS$ in $TU$.    The ideal $I$ is
acyclic and hence the projection 
$$p_{V}: (TU, \delta ) \rightarrow (TU/I, {\bar \delta })$$ is a quasi-isomorphism. From the decomposition $U = V \oplus S \oplus \delta_0 S$, we deduce an
isomorphism of DG  algebras
$ TU/I \cong TV $ which carries  a differential $d_V$ on $TV $ and  defines a
surjective quasi-isomorphism of DG algebras:
$$
\xymatrix{
(TU, D ) \ar@{->>}[rr]_{\simeq} ^{p_V} &&(T^aV, d_V) }\,,
$$
which extends  $\id_V$. The  strict lifting lemma \cite[Proposition 4.4]{ACC} yields  a  section $\varphi_V$ of $p_V$ which is a homomophism of DG algebras.  The composite  $\varphi_A=\alpha_A \circ  \varphi_V$  is a quasi-isomorphism and hence a minimal model.

It is obvious, by the construction that $V$  is isomorphic to $H(U, \delta_0) \cong
s^{-1} H^+\bB (A, d_A)$.
 \end{proof}
 
 \begin{Rem} The DG algebra $(TV, d_V)$ depends on the choice of a direct factor $V$ of $U$. It is easy to see that   another choice  $V'$ provides a minimal model 
 
 $\varphi'_A: (TV', d_{V'})\lra (A, d_A)$ and  the strict lifting lemma yields an isomorphism 
 
 $\varphi: (TV', d_{V'})\lra (TV, d_{V})$ such that $\varphi'_A\circ \varphi\simeq_{\bf DA}\varphi_A$. 
 \end{Rem}

 \begin{Cor}{\sl Let $(A,d_A)$ be a supplemented DG algebra  as in paragraph 2.1 and let $\alpha\in H^*(A, d_A)$ such that 
 
 $\sigma(\alpha)\neq 0$, 
 then there exists a minimal model
 
  $\varphi_A:(TV, d_V)\to (A, d_A)$ such that
 $V\cong \mathbb F_pv\oplus  W$,
 $d_Vv=0$  and 
 $H^*\varphi_A(v)=a$ where $a$ represents $\alpha$.} 
 \end{Cor}   
\begin{proof} Let $\alpha\in H^*(A, d_A)\smallsetminus\{0\}$ such that $\sigma(\alpha)\neq 0$. 
  By Proposition 2.4.1, if $\varphi_A: (TV, d_V)\to \mathcal A$ is a minimal model, then $V\cong s^{-1}H^*\bB (A, d_A)$.

Since $\sigma(\alpha)\neq 0$, $H^*\bB (A, d_A) =\mathbb F_p\sigma(\alpha)\oplus W'$, so that $s^{-1}H^*\bB (A, d_A)= F_ps^{-1}\sigma(\alpha)\oplus s^{-1}W'\cong V$. Set $v= s^{-1}\sigma(\alpha)$ and $W=s^{-1}W'$. Now $\sigma (\alpha) = \mbox{cls}([a])$ where $a$ is a representative of $\alpha$.This implies that $\sigma (\alpha)$ is primitive and hence $d_Vv=0$.

By the  commutativity of diagram (12), $\alpha_A(\langle[a]\rangle) =\varphi_{\mathcal A}\circ p_V(\langle[a]\rangle)= \varphi_{\mathcal A}(v) =a$. 
\end{proof}

\vspace{2mm}
\begin{Lem}{\sl Let $\varphi_A: (TV, d_V)\to (A, d_A)$ be a minimal model. If $\alpha\in H^*(A, d_A)$ is such that $\alpha=H^*\varphi_A(\mbox{cls}(v))$ where $v$ is a cocycle in   $ V\smallsetminus\{0\}$, then $\alpha$ is a generator of $H^*(A, d_A)$ and  $\sigma(\alpha)\neq 0$.
}\end{Lem}
\begin{proof}The bar construction functor preserving quasi-isomorphisms, $\bB \varphi_A: \bB (TV, d_V)\to \bB (A, d_A)$ is a quasi-isomorphism. Note that $\mathbb F_p \oplus sV$ is DG algebra with trivial product and trivial differential and the homomorphism 

$h:\bB (TV, d_V)\to \mathbb F_p\oplus sV$ defined by

$$h([v_1|..|v_n])=\left\{
\begin{array}{ll}
0&\mbox{if}\quad n>1\quad \mbox{or} \quad n=1\quad \mbox{and}\quad v_1\in T^{\geq 2}V\\
sv_1&\mbox{if}\quad n=1 \quad \mbox{and}\quad v_1\in V
 \end{array}
  \right.
$$ is a quasi-isomorphism (\cite{FHT2}, Proposition 19.1).

Let $a$ be a representative of $\alpha$, then $$\sigma(\alpha)= \mbox{cls}([a])=H^*\bB \varphi_{\mathcal A}(\mbox{cls}([v]))=H^*\bB \varphi_{\mathcal A}\circ (H^*h)^{-1}(sv)\neq 0$$ since $H^*\bB \varphi_A$ and $H^*h$ are isomorphisms.
\end{proof}

\section{ Proof of Theorem 1}
\subsection{First  step}

  \begin{Defs} Suppose that $(\theta_n)_{n\geq 1}$ is an infinite \emph{Kraines sequence} in $(A, d_A)$ starting at $\theta=\theta_1$ and 
consider $[\theta_1|..|\theta_k]\in \bB_kA$,  we define 

$$
\delta ^{-1}_i[\theta_1|..|\theta_k] = 
\sum_{j=1}^{i-1}[\theta_1|..| \theta_i| \theta_{i-j}|\theta_{i+1}|..|\theta_k]
$$and
$$\delta ^{-1}[\theta_1|..|\theta_k] = \sum_{i=1}^k \delta ^{-1}_i[\theta_1|..|\theta_k].$$
 Set \begin{eqnarray} \theta(0)=1\quad  \theta_{n,1}:= [\theta_n];\quad  \theta_{n,\ell+1}=  \delta ^{-1}\theta_{n,\ell}, \ell\geq 1,  \theta(n)= \sum_{\ell=1}^n \theta_{n,\ell}.\end{eqnarray}

\end{Defs}

\begin{Lem}(See also \cite{BT3}) \begin{enumerate}\item For every $n\geq 1$, $ \delta  \theta(n)=0$.
 \item   The reduced coproduct  on  $\bB\mathcal A$ satistisfies: 
  
 $$ \bar \Delta  \theta(n) =\sum_{n_1+n_2=n}  \theta(n_1)\otimes  \theta(n_2),  n_1, n_2\geq 1.$$ 	
\item  If    $\sigma(\mbox{cls}( \theta))\neq 0$, then  for $n\geq 1$,    $\mbox{cls}(\theta(n))\neq 0$. \end{enumerate}
 \end{Lem}

 \begin{proof}   Let $( \theta_n)_{n\geq 1}$ be an infinite \emph{Kraines  sequence} in $ (A, d_A)$.
Since each element $ \theta_n$ has odd degree, the differential $ \delta:= \delta'+ \delta''$ on the Adams  reduced   bar construction $\bB(A, d_A)$ defined at (2) is such that:
\begin{eqnarray}
    \renewcommand{\arraystretch}{1.5}
    \begin{array}{rll}
    \delta'  ( [ \theta_{i_1}| \theta_{i_2}|\cdots| \theta_{i_k}])& =  -\sum _{j=1}^k   [ \theta_{i_1}| \theta_{i_2}|\cdots|d_{ A} \theta_{i_j}|\cdots| \theta_{i_k}]\\ 
     \delta''([ \theta_{i_1}| \theta_{i_2}|\cdots| \theta_{i_k}])&= \sum _{j=2}^k 
  [ \theta_{i_1}| \theta_{i_2}|\cdots| \theta_{i_{j-1}} \theta_{i_j}|\cdots |
    \theta_{i_k}] 
        \end{array}
     \renewcommand{\arraystretch}{1} 
   \end{eqnarray}

  \begin{enumerate}
\item From  (14)  and the definition of a \emph{ Kraines sequence}, we have  the following equalities:\\
$ \delta'' \theta_{n,1}=  \delta''[ \theta_n] =0$ and  $ \delta' \theta_{n,1}=  \delta'[ \theta_n]  = -\sum _{i=1}^{n-1} [ \theta_i \theta_{n-i}]=- \delta'' \theta_{n,2} $,
\\
 $ \delta' \theta_{n,n}=  \delta'[ \theta_1| \theta_1|\cdots| \theta_1]=0$ and \\ $ \delta'' \theta_{n,n}=   \delta''[ \theta_1| \theta_1|\cdots| \theta_1] = \sum _{i=2}^k 
  [ \theta_1| \theta_1|\cdots| \theta_1^2|\cdots | \theta_1]  = - \delta' \theta_{n,n-1})$.
  
   For $\ell \geq 1$, we deduce from the definition (13) 
\begin{eqnarray}
     \delta' \theta_{n,\ell} =  -  \delta'' \theta_{n,\ell+1}\,, \qquad 1\leq \ell\leq n-1\,.   \end{eqnarray}
  
Applying  formula (15),
$$
\begin{array}{ll}
\delta  \theta(n) &= \delta' \theta(n) + \delta'' \theta(n) 
= \sum_{\ell=1}^n \delta' \theta_{n,\ell} + \sum_{\ell=1}^n \delta'' \theta_{n,\ell}\\
&=   \delta' \theta_{n, n} +\delta'' \theta_{n, 1}  =0.
\end{array}
$$
That is,  each $ \theta(n)$ is a cocycle in $\bB(A, d_A)$.

 \item  First observe that $ \bar \Delta  a_{n, 1}=   \bar\Delta [ \theta_n]=0$. For $\ell\geq 2$, applying (3),
$$
\renewcommand{\arraystretch}{0.3}
\begin{array}{ll}
 \bar \Delta  \theta(n)&=
\displaystyle \sum_{\ell=2}^n   \bar\Delta  \theta_{n,\ell} = \sum_{\ell=2}^n  \hspace{ 1mm} \displaystyle \sum_{i_1+...+i_\ell=n}  \hspace{ 1mm}   \bar\Delta [ \theta_{i_1}| \theta_{i_2}|\cdots \vert  \theta_{i_\ell}]  \\
 &= \displaystyle \sum_{\ell=2}^n  \sum_{i_1+...+i_\ell=n} \displaystyle  \sum _{j=1}^{\ell-1}     [ \theta_{i_1}\vert\cdots \vert  \theta_{i_j}]\otimes    [ \theta_{i_{j+1}}\vert\cdots \vert  \theta_{i_\ell}]     \\ 
 &=  \displaystyle  \sum_{\ell=2}^n \displaystyle   \sum _{j=1}^{\ell-1}  \sum_{n_1+n_2=n}\displaystyle  \sum_{ 
 \begin{array}{ll}
  i_1+ \cdots+i_j=n_1 \\ i_{j+1} \cdots+ i_\ell=n_2 \end{array}} \hspace{-4mm}  [ \theta_{i_1}\vert\cdots \vert  \theta_{i_j}]\otimes  [ \theta_{i_{j+1}}\vert\cdots \vert  \theta_{i_\ell}]  \\
&= \displaystyle  \sum_{\ell=2}^n \displaystyle   \sum _{j=1}^{\ell-1}  \sum_{n_1+n_2=n}   \theta_{n_1, j}\otimes  \theta_{n_2, \ell-j}\\
 &=  \displaystyle \sum_{n_1+n_2=n}   \displaystyle  \sum_{\ell=2}^n    \hspace{3mm}  \displaystyle   \sum _{  \renewcommand{\arraystretch}{0.1} \begin{array}{ll} \ell_1+\ell_2=\ell \\ 1\leq \ell_1\leq n_1\\1\leq \ell_2\leq n_2 \end{array}  \renewcommand{\arraystretch}{0.5}}  \hspace{-3mm}     \theta_{n_1, \ell_1}\otimes  \theta_{n_2, \ell_2}
 =\displaystyle   \sum _{n_1+n_2=n}  \theta(n_1) \otimes  \theta(n_2).\\
 \renewcommand{\arraystretch}{1}

 \end{array}$$

 \item We  proceed by induction on $n$.
 
 By hypothesis,  $\mbox{cls}(\theta(1))=\sigma(\mbox{cls}(\theta))\neq 0$.   
 
 Suppose that we have proved that the cohomology classes of  $\theta(1), \cdots, \theta(n-1)$ are not trivial.   By point 2 above, 
$H_\ast \bar \Delta\left( \mbox{cls } \theta(n)\right) = \sum _{n_1+n_2=n} \mbox{cls } \theta(n_1) \otimes \mbox{cls } \theta(n_2) \neq 0$.  Thus $\mbox{cls }( \theta(n)) \neq 0.$
 \end{enumerate} 
 
 \end{proof}
\begin{Lem}{\sl If $(\theta_n)_{n\geq 1}$ is an infinite \emph{Kraines sequence } in $(A, d_A)$ starting at $\theta$
 If  there exists a cocycle  of even degree $ b\in  A$,  then there exists an infinite sequence $( b_n)_{n\geq 0}$ in $ A$  such that 
$ b_1 =b$ and for $n\geq 1$,
\begin{eqnarray}d_A b_{n+1} = 
\sum_{i=1}^{n}( \theta_i b_{n+1-i}- b_i \theta_{n+1-i}).
\end{eqnarray}

}\end{Lem}
\begin{proof}

We set $$ b_1= b; \quad \ b_{i+1}:= \theta_i\smile_1 b_1, \quad i\geq 1.$$
Indeed, keeping in mind that $\mbox{deg}( b_1)$ is even and $\mbox{deg}( a_i)$ is odd, $$\begin{array}{ll}
d_A b_{n+1}&= d_A( \theta_n\smile_1 b_1)= \theta_n b_1- b_1 \theta_n-(\sum_{i=1}^{n-1} \theta_i \theta_{n-i})\smile_1 b_1\\
&=\theta_n b_1- b_1 \theta_n- \sum_{i=1}^{n-1} (-1)^{\mbox{deg}(\theta_i)} \theta_i( \theta_{n-i}\smile_1 b_1)\\
&- \sum_{i=1}^{n-1} (-1)^{\mbox{deg}( \theta_{n-i})\mbox{deg}(b_1)}( \theta_i\smile_1 b_1) \theta_{n-i}\\
&= \theta_n b_1- b_1 \theta_n +\sum_{i=1}^{n-1}  \theta_i b_{n+1-i}-\sum_{i=1}^{n-1}  b_{i+1} \theta_{n-i}. \end{array}$$
We make a change of variable in the second sum: $k=i+1$
and obtain that 
$$d_A b_{n+1}=\sum_{i=1}^{n}( \theta_i b_{n+1-i}- b_i \theta_{n+1-i}).$$

\end{proof}

\vspace{3mm}
\begin{Def}

\item  Consider $[ c_1|..| c_k]\in \bB_k( A, d_A)$ with $ c_i=  \theta_i$ or $ c_i= b_i$, we define 

$$
\delta ^{-1}_i[ c_1|..| c_k] = 
\left\{ 
\begin{array}{ll}
\sum_{j=1}^{i-1}[ c_1|..| \theta_i| \theta_{i-j}| c_{i+1}|..| c_k] &\mbox{ if }  c_i= \theta_i\\
\sum_{j=1}^{i-1}[ c_1|..| b_i| \theta_{i-j}| c_{i+1}|...| c_k]-[ c_1|..| \theta_i| b_{i-j}| c_{i+1}|..| c_k] &\mbox{ if } c_i= b_i\,,
\end{array} \right.
$$
$$ \delta ^{-1}[ c_1|..| c_k] = \sum_{i=1}^k \delta ^{-1}_i[ c_1|..| c_k]$$ and $$ \delta ^{-1}(\sum_{i\in I}\lambda_i[ c_{i_1}|..| c_{i_k}])= \sum_{i\in I}\lambda_i \delta ^{-1}[c_{i_1}|..| c_{i_k}],\quad  \lambda_i\in \mathbb F_p.$$

 Set $$   b_{n,1}:= [ b_n];\quad   b_{n,\ell+1}=  \delta ^{-1} b_{n,\ell}, \ell\geq 1,   b(n)= \sum_{\ell=1}^n  b_{n,\ell}.$$

\end{Def}
\begin{Lem}{\sl For every $n\geq 1$,
\begin{enumerate}
\item $\delta b(n)=0$.
\item The reduced  coproduct satisfies $$\bar \Delta b(n)=\sum_{n_1+n_2, n_1, n_2\geq 1}(b(n_1)\otimes \theta(n_2)+\theta(n_1)\otimes b(n_2)).$$
\item If $\sigma(\mbox{cls}(b))\neq 0$, then $\mbox{cls}(b(n))\neq 0$ for every $n\geq 1$.
\end{enumerate}
}\end{Lem}
\begin{proof}

\begin{enumerate}
\item  With the definition  3.1.4, keeping in mind that $\deg(\theta_i)$ is odd and $\deg(\hat b_i)$ is even,  a direct computation yields  \begin{eqnarray} \delta'' b_{n,1}=  0 , \quad \delta'_0\hat b_{n,n}=0,\end{eqnarray}  

  \begin{eqnarray}\delta'' b_{n,\ell+1} =\delta''(\delta^{-1}b_{n,\ell})=   -  \delta'_0b_{n,\ell}\,, \qquad 1\leq \ell\leq n-1.\end{eqnarray}
 From (17) and (18), we deduce that \begin{eqnarray} \delta b(n)=0\end{eqnarray}
\item   We observe that $ \bar \Delta b_{n, 1}=   \bar \Delta [ b_n]=0$. For $\ell\geq 2$ and applying (6), 
$$\begin{array}{ll}
\bar  \Delta \hat b(n)&=
\displaystyle \sum_{\ell=2}^n     \bar \Delta b_{n,\ell} =  \displaystyle \sum_{\ell=2}^n  \sum_{i_1+...+i_\ell=n} \displaystyle  \sum _{j=1}^{\ell-1}   [ c_{i_1}\vert\cdots \vert \hat c_{i_j}]\otimes    [ c_{i_{j+1}}\vert\cdots \vert  c_{i_\ell}]     \\ 
 &=  \displaystyle  \sum_{\ell=2}^n \displaystyle    \sum_{n_1+n_2=n}\displaystyle  \sum_{ 
 \begin{array}{ll}
  i_1+ \cdots+i_j=n_1 \\ i_{j+1} \cdots+ i_\ell=n_2 \end{array}} \hspace{1mm}\sum _{j=1}^{\ell-1}  [ c_{i_1}\vert\cdots \vert  c_{i_j}]\otimes  [ c_{i_{j+1}}\vert\cdots \vert  c_{i_\ell}]  \\
&= \displaystyle  \sum_{\ell=2}^n \displaystyle    \sum_{n_1+n_2=n}  \sum _{j=1}^{\ell-1}  \theta_{n_1, j}\otimes b_{n_2, \ell-j}+b_{n_1, j}\otimes  \theta_{n_2, \ell-j}\\
 &=  \displaystyle \sum_{n_1+n_2=n}   \displaystyle  \sum_{\ell=2}^n    \hspace{1mm}  \displaystyle   \sum _{  \renewcommand{\arraystretch}{0.1} \begin{array}{ll} \ell_1+\ell_2=\ell \\ 1\leq \ell_1\leq n_1\\1\leq \ell_2\leq n_2 \end{array}  \renewcommand{\arraystretch}{0.5}}  \hspace{1mm}     \theta_{n_1, \ell_1}\otimes b_{n_2, \ell_2}+b_{n_1, \ell_1}\otimes  \theta_{n_2, \ell_2}\\
 &
 =\displaystyle   \sum _{n_1+n_2=n}  \theta(n_1) \otimes b(n_2)+ b(n_1) \otimes  \theta(n_2).\\
 \renewcommand{\arraystretch}{1}

 \end{array}$$

\item  We proceed by induction on $n$ to prove that $\mbox{cls}(b(n))\neq 0$.

By hypothesis,   $\mbox{cls}(b(1))=\sigma(\mbox{cls}(b))\neq 0$.

 Suppose proved that $\mbox{cls}(b(k))\neq 0$ for $1\leq k\leq n$.

 Recall that by Lemma 3.1.2-3, $\mbox{cls}(\theta(n))\neq 0$ for every $n\geq 1$.
 By point 2 above,  
  $  \bar \Delta \mbox{cls}( b(n+1)) =\sum_{n_1+n_2=n+1} \mbox{cls}(\theta(n_1))\otimes \mbox{cls}( b(n_2))+\mbox{cls}( b(n_1))\otimes \mbox{cls}( \theta(n_2))\neq 0$  by induction hypothesis. Thus $\mbox{cls}( b(n+1))\neq 0$.\end{enumerate} 
\end{proof}

\begin{Pro}{\sl  Let  $\beta\in H_1^*(A, d_A)\smallsetminus\{0\}$, $\beta=s\alpha$ such that $\sigma(\alpha)\neq 0$.
  If    there exists an infinite \emph{Kraines sequence} $( \theta_n)_{n\geq 1}$ in $(A, d_A)$ starting at $ \theta$  such that $\sigma(\mbox{cls}(\theta))\neq 0$,
  then, 
 there exist two infinite sequences $(\chi_n)_{n\geq 1}$ and $(\psi_n)_{n\geq 1}$ in $H^*\bB (A, d_A)\smallsetminus\{0\}$ such that

  for every $n\geq 1$, $\deg(\chi_n)$ is odd and $\deg(\psi_n)$ is even. 
 }\end{Pro}

 \begin{proof} 
 \begin{enumerate}
 \item If $\deg(\alpha)$ is odd, then $\deg(\beta)$ is even. We  choose $b$ a representative of $\beta$. Since  $\beta=s\alpha$, by Lemma 2.3.7, $\sigma (\mbox{cls}(b))=\sigma (\beta)=\mbox{cls}([b])\neq 0$.
 \item If $\deg(\alpha)$ is even, we choose $b$ a representative of $\alpha$. By assumption $\sigma(\alpha)\neq 0$.\end{enumerate}Therefore,
 \begin{enumerate}
 
\item We apply Lemma 3.1.2 and obtain an infinite sequence $(\theta(n))_{n\geq 1}$ in $\bB A$ of even degrees elements such that for every $n\geq 1$, $\mbox{cls}(\theta_n)\neq 0$. We set $\psi_n = \mbox{cls}(\theta(n))$.
\item We apply Lemma 3.1.5 and obtain an infinite sequence $(b(n))_{n\geq 1}$  in $\bB A$ of odd degrees elements such that for every $n\geq 1$, $\mbox{cls}(b(n))\neq 0$. We set $\chi_n = \mbox{cls}(b(n))$.   

\end{enumerate}\end{proof}
 \vspace{3mm}
 \begin{Rem} From Lemma 3.1.2-2 and Lemma 3.1.5-2 we deduce that, for every $n\geq 1$:
$$  \begin{array}{ll}
  \Delta \psi_n &=\sum_{n_1+n_2}\psi_{n_1}\otimes \psi_{n_2}\\
  \Delta \chi_n&=\sum_{n_1+n_2}(\chi_{n_1}\otimes \psi_{n_2}+\psi_{n_1}\otimes \chi_{n_2}).\end{array}$$
 \end{Rem}
 \subsection{Second step}
In this section, we prove the existence of an infinite \emph{Kraines sequence} in $(A, d_A)$, the DG algebra of normalized  singular cochains on a simply connected finite  CW complex $X$ with coefficients on $\mathbb F_p$ .
 \begin{Lem}{\sl Suppose that $(A, d_A)$ is the DG algebra of normalized  singular cochains on a simply connected finite  CW complex $X$ with coefficients on $\mathbb F_p$.

Let   $\alpha\in H^{2m+1}(A, d_A)\smallsetminus\{0\}$, $m\geq 1$. There exists a cocycle $\theta$  in $A$ and an infinite \emph{Kraines sequence} in $A$ starting at $\theta$.
}\end{Lem}  
\begin{proof}
Let $\alpha\in H^{2m+1}(A, d_A)\smallsetminus\{0\}$ represented by $a$. 
  
  \begin{enumerate}
  \item If there exists an infinite \emph{Kraines sequence} in $(A, d_A)$ starting at $a$, we set $\theta=a$. 
  
  \item Suppose that every \emph{Kraines sequence} staring at $a$ is finite.
  \begin{enumerate} 
  \item Observe that  $\alpha^2=\mbox{cls}(a)^2=0$. 
  
  For $p\geq 3$ this is a consequence of the graded commutativity of the product on $H^*(A, d_A)$.  
  
   For $p=2$, it derives from  Adem's relations on the decomposability of  the Steenrod operation $Sq^k$ for $k>1$, $k$  odd.

In both cases, there exists $a_2\in A$ such that $d_Aa_2=a^2$.

\item Let $N\geq 2$ be the greatest integer such that  $(a_n)_{1\leq n\leq N}$ is a \emph{Kraines sequence} starting at $a$, that is, for $1\leq n\leq N$, $d_Aa_n=\sum_{i=1}^{n-1}a_ia_{n-i}$ and  $\mbox{cls}(\sum_{i=1}^{N} a_i a_{N+1-i})\neq 0.$ 

 For $1\leq n\leq N$, there exist $\hat a_n\in \hat A$ and $\hat \zeta_n\in ker\hspace{1mm}\mbox{red}_p$ such that  \begin{enumerate}
 \item $\mbox{red}_p(\hat a_n)=a_n$ and $d_{\hat A}\hat a_n =\sum_{i=1}^{n-1}\hat a_i\hat a_{n-i}+\hat \zeta_n$,
 \item $d_{\hat A}\sum_{i=1}^{N}\hat a_i\hat a_{N+1-i}=\sum_{i=1}^N(\hat \zeta_i\hat a_{N+1-i}-\hat a_i\hat \zeta_{N+1-i})$.
 \end{enumerate}
 \item We prove  the existence of  $\hat \zeta_{N+1}\in ker\hspace{1mm}\mbox{red}_p$ such that $d_{\hat A}(\sum_{i=1}^{N}\hat a_i\hat a_{N+1-i}+\hat \zeta_{N+1})=0$.
 \begin{enumerate}
 \item Recall that for  $1\leq n\leq N$,  $\hat \zeta_n\in ker\hspace{1mm}\mbox{red}_p$ and set $\hat \zeta_n=p^{\epsilon_n}\hat \zeta'_n $. Note $\epsilon_N=\inf\{\epsilon_n \geq 1, 1\leq n\leq N\}$ and set $\hat \zeta_n=p^{\epsilon_{N}}\hat z_n.$
 
  Thus $d_{\hat A}\sum_{i=1}^N\hat a_i\hat a_{N+1-i}=p^{\epsilon_N}\sum_{i=1}^N(\hat z_{N+1-i}\hat a_i-\hat a_i\hat z_{N+1-i})$.
 \item Recall that  $( A_{\epsilon_N}, d_{A_{\epsilon_N}})=(\hat A, d_{\hat A})\otimes _{\mathbb Z}\mathbb Z/p^{\epsilon_N}\mathbb Z$  and $\mbox{red}_{p^{\epsilon_N}}: \hat A\to  A_{\epsilon_N}$ the reduction mod$_{p^{\epsilon_N}}$. 
 
\item  Observe that for every $i$, $1\leq i\leq N$, $\mbox{red}_{p^{\epsilon_N}}(\hat \zeta_i)=0$.
 
\item If $\hat y\in \hat A$, we note $\bar y= \mbox{red}_{p^{\epsilon_N}}(\hat y)$. Thus $\bar a=\mbox{red}_{p^{\epsilon_N}}(\hat a)$ is a cocycle in  $( A_{\epsilon_N}, d_{ A_{\epsilon_N}})$ and $\sum_{i=1}^N\bar a_i\bar a_{N+1-i}$ is a cocycle in $( A_{\epsilon_N}, d_{ A_{\epsilon_N}})$.
 
 \item We prove that $\beta_{\epsilon_N}(\mbox{cls}(\sum_{i=1}^N\bar a_i\bar a_{N+1-i}))=0$ where $\beta_{\epsilon_N}: H^q( A_{\epsilon_N}, d_{{\epsilon_N}})\to H^
  {q+1}( A_{\epsilon_N}, d_{{\epsilon_N}})$ is the Bockstein homomorphism associated to the short exact sequence $0\to\mathbb Z/p^{\epsilon_{N}}\mathbb Z\to \mathbb Z/p^{2\epsilon_{N}}\mathbb Z\to \mathbb Z/p^{\epsilon_{N}}\mathbb Z\to 0.$

   For this purpose,  we define $ 1\leq n\leq N$ in $A_{\epsilon_N}$,
  $$\bar X_1=0, \quad \bar X_{n+1}=\sum_{i=1}^n\bar  a_i\smile_1(\bar z_{n-i}- \bar X_{n-i}),$$
   
  and prove by induction on $n$ that $$d_{A_{\epsilon_N}}\bar X_n=\sum_{i=1}^{n-1}(\bar a_i\bar z_{n-i}-\bar z_i\bar a_{n-i})=d_{A_{\epsilon_N}}\bar z_n.$$ 
  
  Indeed, since $d_{\hat A} \hat a=\hat \zeta_1$ with $\hat \zeta_1\in ker\hspace{1mm}\mbox{red}_p$ and , $d_{\hat A}\hat \zeta_1=0$. By construction, if $\hat \zeta_1\neq 0$, then $\epsilon_1\geq \epsilon_N$ and $\hat \zeta_1=p^{\epsilon_N}\hat z_1$. Since $\hat A$ is torsion free and  $d_{\hat A}\hat z_1=0$, then $d_{A_{\epsilon_N}} \bar z_1=0 =d_{A_{\epsilon_N}} \bar X_1$. 
  
  Suppose proved that $d_{A_{\epsilon_N}}\bar X_k=\sum_{i=1}^{k-1}(\bar a_i\bar z_{k-i}-\bar z_i\bar a_{k-i})=d_{A_{\epsilon_N}}\bar z_k$ for $1\leq k\leq n$ and suppose that $n+1\leq N$. Thus 
  $$\begin{array}{ll}d_{A_{\epsilon_N}}\bar X_{n+1}&=\sum_{i=1}^nd_{A_{\epsilon_N}}(\bar  a_i\smile_1(\bar z_{n+1-i}- \bar X_{n+1-i}))\\
 &=\sum_{i=1}^n(\bar a_i( \bar z_{n+1}-\bar X_{n+1-i})-(\bar z _{n+1}-\bar X_{n+1-i})\bar a_i)\\
& -\sum_{i=1}^n\sum_{j=1}^{i-1}(\bar a_j\bar a_{i-j}+\bar \zeta_i)\smile_1(\bar z_{n+1-i}-\bar X_{n+1-i}))\\
&+\sum_{i=1}^n\bar a_i\smile_1(d_{\hat A}(\bar z_{n+1-i}-\bar X_{n+1-i})\\
&=\sum_{i=1}^n(\bar a_i\bar z_{n+1-i}-\bar z_{n+1-i}\bar a_i)- \sum_{i=1}^n(\bar a_i\bar X_{n+1-i}-\bar X_{n+1-i}\bar a_i)\\
&+\sum_{i=1}^n\sum_{j=1}^{i-1}\bar a_j(\bar a_{i-j}\smile_1(\bar z_{n+1-i}-\bar X_{n+1-i}))\\
&-\sum_{i=1}^n\sum_{j=1}^{i-1}(\bar a_{i-j}\smile_1(\bar z_{n+1-i}-\bar X_{n+1-i}))\bar a_j \\
&\mbox{since}\quad d_{A_{\epsilon_N}}(\bar z_{n+1-i}-\bar X_{n+1-i})=0\quad \mbox{by induction hypothesis}\\
& \mbox{and} \quad \bar \zeta_i=0\quad \mbox{for}\quad 1\leq i\leq n+1.\\
  
 \end{array}$$ By a direct checking, using the definition of $\bar X_k$, we obtain that:
 $$\begin{array}{ll}
 \sum_{i=1}^n\sum_{j=1}^{i-1}\bar a_j(\bar a_{i-j}\smile_1(\bar z_{n+1-i}-\bar X_{n+1-i}))&=
 \sum_{i=1}^n\bar a_i\bar X_{n+1-i},\\
 \sum_{i=1}^n\sum_{j=1}^{i-1}(\bar a_{i-j}\smile_1(\bar z_{n+1-i}-\bar X_{n+1-i}))\bar a_j&=\sum_{i=1}^n\bar X_{n+1-i}\bar a_i
\end{array}$$  and  hence, $- \sum_{i=1}^n(\bar a_i\bar X_{n+1-i}-\bar X_{n+1-i}\bar a_i)\\
+\sum_{i=1}^n\sum_{j=1}^{i-1}\bar a_j(\bar a_{i-j}\smile_1(\bar z_{n+1-i}-\bar X_{n+1-i}))\\
-\sum_{i=1}^n\sum_{j=1}^{i-1}(\bar a_{i-j}\smile_1(\bar z_{n+1-i}-\bar X_{n+1-i}))\bar a_j=0$.

 Therefore $d_{A_{\epsilon_N}}\bar X_{N+1}= \sum_{i=1}^N(\bar a_i\bar z_{n+1-i}-\bar z_{n+1-i}\bar a_i)$ and 
 $$\beta_{\epsilon_N}(\mbox{cls}(\sum_{i=1}^N\bar a_i\bar a_{N+1-i}))=\mbox{cls}(\sum_{i=1}^N(\bar a_i\bar z_{n+1-i}-\bar z_{n+1-i}\bar a_i))=\mbox{cls}(d_{A_{\epsilon_N}}\bar X_{N+1})=0.$$

  Thus by Lemma 2.3.8, there exists $\hat \zeta_{N+1}\in \hat A$ such that

 $d_{\hat A}(\sum_{i=1}^{N}\hat a_i\hat a_{N+1-i}+\hat \zeta_{N+1})=0$.
 \end{enumerate}
 \item We prove that there there exist $\hat a_{N+1}\in \hat A$ and an integer $\epsilon_{N+1}\geq 1$ such that $d_{\hat A}\hat a_{N+1}=p^{\epsilon_{N+1}}(\sum_{i=1}^N\hat a_i\hat a_{N+1-i}+\hat \zeta_{N+1}).$ That is , $\mbox{cls}(\sum_{i=1}^N\hat a_i\hat a_{N+1-i}+\hat \zeta_{N+1})\in \mbox{tor}H^*(\hat A, d_{\hat A})$. 
 \begin{enumerate}
 \item If $\hat \zeta_{N+1}=0$, this is a part of Theorem 15 of \cite{Kr2}. So, there exist $\hat a_{N+1}\in \hat A$ and an integer $\epsilon_{N+1}\geq 1$ such that $d_{\hat A}\hat a_{N+1}=p^{\epsilon_{N+1}}\sum_{i=1}^N\hat a_i\hat a_{N+1-i}$.
 \item Suppose that $\hat \zeta_{N+1}\neq 0$. 
 
  Remark that, from (c)-iii above, $(\bar a_n)_{1\leq n\leq N}$ is a \emph{Kraines sequence} in $(\hat A/p^{\epsilon_N}\hat A, \bar d_{\hat A})$ and by Lemma 2.3.10-2,
  
  $\mbox{cls}(\sum_{i=1}^{N}(\bar a_i\bar a_{N+1-i})\in \mbox{tor}H^*(\hat A/p^{\epsilon_N}\hat A, \bar d_{\hat A})$. By Lemma  
  2.3.10-1, $\mbox{cls}(\sum_{i=1}^{N}(\hat a_i\hat a_{N+1-i})\in \mbox{tor}H^*(\hat A,  d_{\hat A})$.

  Thus,  there exist $\hat a_{N+1}\in \hat A$ and an integer $\epsilon_{N+1}\geq 1$ such that

    $d_{\hat A}\hat a_{N+1}=p^{\epsilon_{N+1}}(\sum_{i=1}^N\hat a_i\hat a_{N+1-i}+\hat \zeta_{N+1})$. 
 \end{enumerate}

\item From (d) above,   
 $a_{N+1}=\mbox{red}_p(\hat a_{N+1})$ is a cocycle in $ A$ and 

$\mbox{cls}(a_{N+1})= s\mbox{cls}(\sum_{i=1}^Na_ia_{N+1-i})$.   
   
\item If there exists an infinite \emph{Kraines sequence } starting at $a_{N+1}$, we set $\theta=a_{N+1}$. If not, we restart the procedure with $a_{N+1}$ until we obtain $\theta$ and  an infinite \emph{Kraines sequence }  $(\theta_n)_{n\geq 1}$ starting at $\theta$, since the cohomology $H^*(A, d_A)$ is finite.

\end{enumerate}\end{enumerate}\end{proof}

\begin{Lem}{\sl  If $\sigma(\alpha)\neq 0$, then $\sigma(\mbox{cls}(\theta))\neq 0$.
   }\end{Lem}
  \begin{proof} It is enough to prove that $\sigma(\mbox{cls}(a_{N+1}))\neq 0$.

 By assumption, $\sigma(\alpha)\neq 0$ and  
  by Corollary 2.4.3, there exists a minimal model $\varphi'_A: (TV', d_{V'})\to (A, d_A)$ where $V'=\mathbb F_pv'\oplus W'$ such that $d_{V'}v'=0$  and $\varphi'_A(v') =a$  where $a$ represents $\alpha$.

 Let $\varphi_{\hat A}:(T\hat V, d_{\hat V})\to (\hat A, d_{\hat A})$ be  a minimal model. The isomorphism (10)  yields a minimal model $\varphi_A:(T V, d_{ V})\to (A, d_A)$.
 
By Remark 2.4.2, we have an isomorphism of DG algebras 

$\varphi: (TV', d_{V'})\lra (TV, d_{V})$ such that $\varphi_A\circ \varphi\simeq_{\bf DA}\varphi'_A$. 
 
 In particular $\varphi_{|V'}:V'\to V$ is an isomorphism. 
 
 Set $v=\varphi_{|V'}(v')\in V$.
 \begin{enumerate}
\item  There exists a finite \emph{Kraines sequence } $(v_n)_{1\leq n\leq N}$ in $TV$ starting at $v$ and such that $\varphi_{\mathcal A}(v_n)= a_n$.

 Indeed, $\varphi_{\mathcal A}(v)= a$ and $v^2$ is a cocycle in $(TV, d_V)$ such that $\varphi_{\mathcal A}(v^2)=a^2=d_Aa_2$. Then there exist $v_2\in TV$ such that $d_Vv_2=v^2$ since $\varphi_A$ is a quasi-isomorphism. Continuing so, we obtain the claimed \emph{Kraines sequence}. Furthermore, $\sum_{i=1}^Nv_iv_{N+1-i}$ is a cocycle in $(TV, d_V)$ and
 
  $\mbox{cls}(\sum_{i=1}^Nv_iv_{N+1-i})=(H^*\varphi_{\mathcal A})^{-1}( \mbox{cls}(\sum_{i=1}^Na_ia_{N+1-i}))\neq 0$.
\item For $1\leq n\leq N$, there exist $\hat v_n\in T\hat V$ and $\hat z_n\in ker\hspace{1mm}\mbox{red}_p$ such that
\begin{enumerate}
\item $\mbox{red}_p(\hat v_n)=v_n$ and $d_{\hat V}\hat v_n =\sum_{i=1}^{n-1}\hat v_i\hat v_{n-i}+\hat z_n$, 
\item  $d_{\hat V}\sum_{i=1}^{N}\hat v_i\hat v_{N+1-i}=\sum_{i=1}^N(\hat z_i\hat v_{N+1-i}-\hat v_i\hat z_{N+1-i})$.\end{enumerate}
  
 \item In point 2-(c) of the proof of Lemma 3.2.1,  we have proved the existence of  $\hat \zeta_{N+1}\in \mbox{ker}\hspace{1mm}\mbox{red}_p$ such that $\sum_{i=1}^N\hat a_i\hat a_{N+1-i} +\hat \zeta_{N+1}$ is a cocycle with non trivial cohomology class. 
 
 Since $\varphi_{\hat{\mathcal A}}$ is a quasi-isomorphism and diagram (11) is commutative, there exists $\hat \chi_{N+1}\in \ker(\varphi_{\mathcal A}\circ\rho_V)$ such that $\sum_{i=1}^N\hat v_i\hat v_{N+1-i} +\hat \chi_{N+1}$ is a cocycle in $(T\hat V, d_{\hat V})$ and $\varphi_{\hat{\mathcal A}}(\sum_{i=1}^N\hat v_i\hat v_{N+1-i} +\hat \chi_{N+1})= \sum_{i=1}^N\hat a_i\hat a_{N+1-i} +\hat \zeta_{N+1}$.
 
\item Since $v=v_1\in V$, $d_V$ strictly increases words length and $d_{\hat V}$ strictly increases words length  modulo $p\hat V$, an easy induction on $n$ shows that $v_n\in V$ and $\hat v_n\in \hat V$.
\item By Lemma 3.2.1-2-(d), there exist $\hat a_{N+1}\in \hat{\mathcal A}$ and an integer $\epsilon_{N+1}\geq 1$ such that $d_{\hat A}\hat a_{N+1}=p^{\epsilon_{N+1}}\sum_{i=1}^N\hat a_i\hat a_{N+1-i} +\hat \zeta_{N+1}$. As $\varphi_{\hat{\mathcal A}}$ is a quasi-isomophism , there exists $\hat v_{N+1}\in T\hat V$ such that $d_{\hat A}\hat v_{N+1}=p^{\epsilon_{N+1}}\sum_{i=1}^N\hat v_i\hat v_{N+1-i} +\hat \chi_{N+1}$ and $\varphi_{\hat{\mathcal A}}(\hat v_{N+1})=\hat a_{N+1}$.

Since $d_{\hat V}$ strictly increases words length  modulo $p\hat V$, $\hat v_{N+1}\in \hat V$ and hence $v_{N+1}=\mbox{red}_p(\hat v_{N+1})\in  V$. But $v_{N+1}$ is a cocycle in $(TV, d_V)$ and $\varphi_{\mathcal A}(v_{N+1})=a_{N+1}.$ We apply Lemma 2.4.4 and deduce that  $\sigma(\mbox{cls}(a_{N+1}))\neq 0$. 
 \end{enumerate}
\end{proof}

\subsection{End of proof}

  \begin{Cor}{\sl  Suppose that $ (A, d_A)$ is the DG algebra of normalized  singular cochains on a simply connected finite  CW complex $X$ with coefficients in $\mathbb F_p$.

If   $\beta\in H^*_1(A, d_A)\smallsetminus\{0\}$ such that $\beta=s\alpha$ such that $\sigma(\alpha)\neq 0$, then  there exist two infinite sequences $(\chi_n)_{n\geq 1}$ and $(\psi_n)_{n\geq 1}$ in $H^*\bB (A, d_A)\smallsetminus\{0\}$ such that 
 
   for every $n\geq 1$, $\chi_n\in H^{odd}\bB (A, d_A)$ and $\psi_n\in H^{even}\bB (A, d_A)$.}\end{Cor}

  \begin{proof}
  
  If $\beta\in H^*_1(A, d_A)\smallsetminus\{0\}$, then  $\beta\in H^{odd}(A; d_A)$ or $\alpha\in H^{odd}(A; d_A)$.
 
By Lemma 3.2.1, there exists a cocyle $\theta\in A$ and an infinite \emph{Kraines sequence} starting at $\theta$. Since $\sigma(\alpha)\neq 0$, by Lemma 3.2.2, $\sigma(\theta)\neq 0$.

 We apply Proposition 3.1.7.
  \end{proof}

\section{Proof of Theorem 2}

\subsection{Loop space}
 We recall the following result.
 \begin{Pro}(\cite{ACC}, \cite{BT}){\sl There exists a natural isomorphism  $$
H^*(\Omega X;\mathbb F_p) \cong H^*\bB (A, d_A)\,,$$
as graded Hopf algebras.

  }\end{Pro}
  
  \begin{Pro}{\sl If $H^*_1(X; \mathbb F_p)\neq \{0\}$, the sequence $(dim_{\mathbb F_p}H^n(\Omega X; \mathbb F_p))_{n\geq 1}$ grows unbounded.
  }\end{Pro}
\begin{proof}
Let $(\chi_n)_{n\geq 1}$ and $(\psi_n)_{n\geq 1}$  be the sequences of Theorem 1.

We set $$\chi_0 = \psi_0 =\zeta_{0, 0} =1, \quad \zeta _{r, 0}:=\chi_r, \quad \zeta_{0, s}:=\psi_s$$ 
and 
$$ \zeta_{r, s}:=\chi_r\psi_s,\quad r,s \geq 1.$$

Observe that:
\begin{enumerate}
\item For every $n\geq 1$,  $\deg(\chi_n)$ is odd and  
$\deg(\psi_n)$ is even,
\item $(\deg(\chi_n))_{n\geq 1}$ and $(\deg(\psi_n))_{n\geq 1}$ are arithmetic sequences.
\end{enumerate}
We follow the arguments of \cite{Mc2}, paragraph 2, to conclude that $H^*(\Omega X; \mathbb F_p)$ contains a vector space isomorphic, as a vector space, to the polynomial algebra $\mathbb F_p[u, v]$.
\end{proof}

\subsection{Free loop space}
Let $(A, d_{A})$ be
a DG  algebra with coefficients on a commutative ring $\bk$ with unit.

If $\delta$ denotes
the differential of the reduced bar construction $\bB (A, d_{A})$, then the
tensor product $(A, d_{A})\otimes_{\bk}\bB (A, d_{A})$
is a chain complex whose differential is denoted by $d_A\otimes \delta$ defined by $d_A\otimes \delta(a\otimes x)= d_Aa\otimes x+(-1)^{\deg(a)}a\otimes \delta x$ where $a\in A$ and $x\in \bB A$.\\
By definition the Hochschild complex of $(A, d_A)$ with coefficients in the $A$-bimodule $A$
is a pair $(\frak{C}(A), D)$, where $\frak {C}_{\ast}A= \{\frak
{C}_{k}(A)\}_{k\geq 0}$; $\frak {C}_{k}(A)= A\otimes B_{k}A$ and
$$
\begin{array}{ccc}
D(a_0\otimes [a_{1}|a_{2}|\cdot\cdot\cdot|a_{k-1}|a_{k}])&=& (d_{0}-
d_{1})(a_0\otimes [a_{1}|a_{2}|\cdot\cdot\cdot|a_{k-1}|a_{k}]) \\
&+&
d_A\otimes \delta(a_{0}\otimes[a_{1}|a_{2}|\cdot\cdot\cdot|a_{k-1}|a_{k}]),\\
\end{array}
$$
with
$$d_{0}(a_{0}\otimes[a_{1}|a_{2}|\cdot\cdot\cdot|a_{k-1}|a_{k}])=
(-1)^{\deg(a_{0})}a_{0}a_{1}\otimes[a_{2}|\cdot\cdot\cdot|a_{k-1}|a_{k}]$$ and
$$d_{1}(a_{0}[a_{1}|a_{2}|\cdot\cdot\cdot|a_{k-1}|a_{k}])=
(-1)^{(\deg(a_{k})+ 1)(\deg(a_{0})+\cdot\cdot\cdot+ \deg(a_{k-1})+
k-1)}a_{k}a_{0}\otimes[a_{1}|a_{2}|\cdot\cdot\cdot|a_{k-1}].$$
The
Hochschild homology of $(A, d_A)$ with coefficients in  the $A$-bimodule $A$ denoted by
$HH_{\ast}(A, d_A)$ is defined as follows: $HH_{\ast}(A, d_A)=
H_{\ast}(\frak{C}(A), D).$

 When $(A, d_A) = C^*X$, the DG algebra of  normalized cochains   on the topological space $X$, $HH_{\ast}(X):= HH_{\ast}(C^*X)$ is the Hochschild homology of $X$.

When $((A, d_A), \mu_A)$ is a shc DG algebra (in the sense of \cite{Mu}) and $\bk$ a field, regarding \cite[Theorem 1]{BT}, the composite  
    $$
    \xymatrix{
       {\mathfrak C}_\ast  A \otimes  {\mathfrak C}_\ast  A 
       \ar[r]_{sh}  \ar @/^3pc/[rrrrr]^\star &  {\mathfrak C}_\ast ( A\otimes  A)
       \ar[rr]_{ s_{ A\otimes  A}} &&{\mathfrak C}_\ast \bar
       \Omega \bB(A\otimes A)
       \ar[r]_{\quad {\mathfrak C}_\ast  \mu_A} & {\mathfrak C}_\ast  \bar \Omega \bB \mathcal A
      \ar[r]_{\,{\mathfrak C}_\ast  \alpha_A} & {\mathfrak C}_\ast \mathcal A \\
}
  $$
is a  product on ${\mathfrak C}_\ast  A$ such that   $\hH_\ast(A, d_A)$ is a commutative algebra.
Here $s_{ A\otimes  A}$ a linear  section of 
$ {\mathfrak C}_\ast  \bar \Omega \bB  (A\otimes A)\rightarrow {\mathfrak C}_\ast  (A\otimes A).$

 The sequence 
\begin{eqnarray}
 \xymatrix{
(A, d_A)
       \ar[rr] ^{I }&& 
        {\cal C}_*(A, d_A)   \ar@{->>}[rr]^{\Phi}&& \bB (A, d_A)
       } 
       \,,
       \end{eqnarray} is of homomorphisms of DG  algebras, where

$$
 I(a_0\otimes[a_1\vert a_2\vert\cdots\vert a_n]) =  \left\{ \begin{array}{lll} 
 &0 &\mbox{ if  }a_0\in \bar A\\
&[a_1\vert a_2\vert\cdots\vert a_n] &\mbox{ if  } a_0=1. 
\end{array}
\right.\,.$$

\subsubsection{End of the proof}
 We recall the following result.
 \begin{Pro}(\cite{Mu}){\sl The DG algebra of normalized cochains $C^*(X; \bk)$ on $X$ is a shc DG algebra.
 
 }\end{Pro}
\begin{Pro}(See \cite{BT}, \cite{J}){\sl If $(A, d_A)=C^*(X; \bk)$ is the DG algebra of normalized singular cochains  on $X$,   there exists a natural isomorphism  $$
H^\ast(LX;\mathbb F_p) \cong \hH_*(A, d_A)\,,$$
as graded algebras.
  
  }\end{Pro}

\vspace{3mm}
\begin{Rem}
  In the proof of Theorem 1, we have constructed  two sequences of  cocycles   $( \theta(n))_{n\geq 1}$, $(b(n))_{n\geq 1}$ in $\bB( A, d_{\hat A})$ such that 
 $\mbox{cls}(\theta(n))=\psi_n\neq 0$ and   $\mbox{cls}(b(n))=\chi_n\neq 0$.

\end{Rem}

  We set in ${\cal C}_*( A, d_{ A})$,
 $$  z_{r, 0}:=1\otimes b(r),  z_{0, s}:=1\otimes  \theta(s),  z_{r, s}=z_{r, 0}\star  z_{0, s}.$$
 
   \begin{Lem}{\sl The elements $z_{r, s}$ are cycles in 
   ${\cal C}_*( A, d_{ A})$.
   }\end{Lem}  
  \begin{proof} \begin{enumerate}
  \item We first  prove that $( d_0- d_1)(1\otimes  z_{r,0})=( d_0- d_1)(1\otimes  z_{0,s})=0$.  We do it for $b(n)$, the case $\theta(n)$ being similar and more simple.
   
   Indeed, $$
   \begin{array}{l}
   ( d_0- d_1)(1\otimes  z_{r,0}) = ( d_0- d_1)(1\otimes b(r))\\
    =\sum_{l=1}^r\sum_{i_1+...+i_\ell=r} c_{i_1}\otimes[ c_{i_2}| c_{i_3}|....| c_{i_\ell}]\\
    -\sum_{\ell=1}^r\sum_{i_1+...+i_\ell=r} (-1)^{(\mbox{deg}( c_{i_{\ell}})+ 1)(\mbox{deg}( c_{i_1})+..+ \mbox{deg}( c_{i_{\ell-1}})+
\ell-1)} c_{i_l}\otimes[ c_{i_1}| c_{i_2}|....| c_{i_{\ell-1}}].
    \end{array}$$
    Remember that $\deg ( \theta_{i_j})$ is odd and $\deg ( b)$ is even.
      The statement that {\it  for every partition  $\sum_{j=1}^{\ell}i_j =r$ there exists only one $j_0$ such that $ c_{i_{j_0}}= b_{j_0}$  and $c_{i_j}=\theta_{i_j}$ for all $j\neq j_0$}, implies that
         $\deg ( b_{i_j})$ is even and hence   $(\mbox{deg}( c_{i_{\ell}})+ 1)(\mbox{deg}( c_{i_1})+..+ \mbox{deg}( c_{i_{\ell-1}})$ is always even. Thus $( d_0- d_1)(1\otimes  z_{r,0})=( d_0- d_1)(1\otimes  z_{0,s})=0$.
       
 \item  Notice that $D$ is a derivation and $z_{r, s}= z_{r,0}\star z_{s,0}$. As $\delta b(r)=\delta \theta(s)=0$, then $Dz_{r, s}= (d_0-d_1)z_{r, s}$.

 Since $Dz_{r, s}=D(z_{r,0}\star z_{s,0})= (Dz_{r,0})\star z_{0,s} + (-1)^{\deg(z_{r,0})}(z_{r,0})\star Dz_{0,s}$  then  $z_{r, s}$ is a $D$-cycle. \end{enumerate} 
 \end{proof}
  
 \begin{Pro}{\sl  Let $X$ be a finite CW complex.
 
 If $H^*_1(X; \mathbb F_p)\neq \{0\}$, the sequence $(dim_{\mathbb F_p}H^n(LX; \mathbb F_p))_{n\geq 1}$ grows unbounded.
  }\end{Pro}
\begin{proof} With the notations above,  $H^*\Phi (\mbox{cls}(z_{r,s})) =  \zeta_{r, s} \neq 0 $.

 From Proposition 4.1.2 and Proposition 4.2.2, the sequence  $(dim H^n(LX;\mathbb F_p))_{n\geq 1}$  grows unbounded.
 
 \end{proof}
 
\vspace{3mm}

\section{Proof of main  Theorem }

 \subsection{Recollection of some known results }

\begin{Lem} {\it  If $X$ is a  simply connected  CW complex  of finite type  such that   $H^*(X; \mathbb Q)$ has at least two  generators as an algebra, then 
\begin{enumerate}
\item the sequence of the Betti numbers $\{dimH^i(\Omega X; \mathbb \mathbb F_p)\}_{i\geq 1}$ grows unbounded, 
\item the sequence of the Betti numbers $\{dim H^i(LX; \mathbb F_p)\}_{i\geq 1}$ grows unbounded.
 \end{enumerate}}
\end{Lem}

\vspace{2mm}
\begin{proof}
\begin{enumerate}
\item This  is a consequence of a  theorem of \cite{Su} and  the universal coefficients theorem. See also \cite{Mc2} for more details.
\item It is a consequence of a result of \cite{VS} and the universal coefficients theorem.
\end{enumerate}
\end{proof}
The following observation is obvious.
  
  \begin{Lem}{\sl The  assertions 
  \begin{enumerate}
  \item $\mbox{tor}H^*(X; \mathbb Z)\neq \{0\}$,
  \item $H_1^*(X; \mathbb F_p)\neq \{0\}$
  \end{enumerate}
are equivalent where $\mbox{tor}H^*(X; \mathbb Z)$ is the $p$-torsion part of $H^*(X; \mathbb Z)$.

}\end{Lem}

\vspace{3mm}
  \subsection{First  step}
   Let $X$ be a simply connected  finite CW complex.

We now examine  the case when $H^*(X; \mathbb Q)$ has at maximum   one generator as an algebra and $H^*(X; \mathbb F_p)$ has at least two generators. In this case $H_1^*(X; \mathbb F_p)\neq \{0\}$.

\begin{Lem}{\sl  Suppose that  $H^*(X; \mathbb Q)$ has at maximum   one generator as an algebra and 
$H_1^*(X; \mathbb F_p)\neq \{0\}$.

Let $k= \mbox{inf}\{ i\quad \mbox{such that} \quad H^i_1(X; \mathbb F_p)\neq \{0\}\}$.

If $\beta=s\alpha\in H^k_1(X; \mathbb F_p)$, then  $\beta$ and $\alpha$   are generators and $\sigma(\alpha)\neq 0$.
 
}\end{Lem}

 \begin{proof}
 Let $\beta\in  H^k_1(X; \mathbb F_p)$.
 
 \begin{enumerate}
 \item If  $H^*(X; \mathbb Q)=\{0\}$, then $C^*(X; \mathbb F_p)$ is $(k-1)$-connected, $k\geq 2$, and  $\alpha$ and $\beta$ are indecomposable.  By corollary 8.25 of \cite{Mc}, $\sigma(\alpha)\neq 0$ and $\sigma(\beta)\neq 0$.
 
 \item Suppose that  $H^*(X; \mathbb Q)$ has only one generator $w$.  
 \begin{enumerate}
 \item If $k\leq \mbox{deg}( w)$,  then,  again   we apply corollary 8.25 of \cite{Mc} and obtain that $\sigma(\beta)\neq 0$ and $\sigma(\alpha)\neq 0$.
 \item  Suppose that $k> \mbox{deg}( w)$. It is obvious that $\alpha$ and $\beta$ are indecomposable.

 Observe that in the   Leray-Serre spectral for the path-loop fibration,  and since $X$ is simply connected,   $\alpha$ and $\beta$ are transgressed. Thus following \cite{MT}VII-2.6, $\sigma(\alpha)\neq 0$ and $\sigma(\beta)\neq 0$.
 
 \end{enumerate}
 \end{enumerate}
 
 \end{proof}

\begin{Pro}{\sl  Let $X$ be a simply connected  finite CW complex  such that 

 $H^*(M; \mathbb Q)$ has at most one generator as an algebra and   $H^*(X; \mathbb F_p)$ admits at least two generators as an algebra,
 then 
\begin{enumerate}
\item the sequence  $(dim H^n(\Omega X; \mathbb F_p))_{n\geq 1}$ grows unbounded,  
 \item the sequence $(dim H^n(LX; \mathbb F_p))_{n\geq 1}$ grows unbounded.  
\end{enumerate}

}\end{Pro}

\begin{proof}

 Let $k= \mbox{inf}\{i\quad \mbox{such that}\quad H_1^i(X; \mathbb F_p)\neq \{0\}$ and $\beta=s\alpha\in H_1^k(X; \mathbb F_p)$.

 By Lemma  5.2.1, $\beta$ and $\alpha$ are generators, $\sigma(\beta)\neq 0$  and $\sigma(\alpha)\neq 0$.
 
  We then apply   Theorem 2  to conclude.

\end{proof}

\subsection{End of the proof}

Let $\bk$ be a field of characteristic $p\geq 2$.  Then $\bk$ is a $\mathbb F_p$ vector space via  the canonical inclusion $\mathbb F_p \subset \bk$ and for each $i$
$$H^i(M, \bk)=\left(H^i(M;\mathbb F_p)\otimes \bk\right)\oplus s\mbox{Tor}(H^{i+1}(M;\mathbb F_p), \bk)\,.$$ The  second summand in the right hand term  being zero, the above results extend for any field $\bk$.

\footnotesize{

}
 \end{document}